\documentclass[xcolor,table,nobib]{tufte-handout}

\usepackage{matyi}
\theoremstyle{remark}
\newtheorem*{remarks}{Remark}
\usepackage{float}
\usepackage{tkz-euclide}
\usepackage{fbb}

\usepackage[backend=biber,autocite=footnote,style=authoryear]{biblatex}

\newcommand{\rr}{\raggedright}
\newcommand{\cq}{\coloneqq}
\title{Sublacunary sequences that are strong sweeping out}
\author{Sovanlal Mondal\thanks{Supported by the National Science
    Foundation, grant number DMS-1855745.}, Madhumita Roy and M\'at\'e
  Wierdl\\
  \vspace{0.5 cm}
  \normalsize Department of Mathematical Sciences,
  University of Memphis\\
  Memphis, TN 38152, USA \\
  email \href{mailto:smondal@memphis.edu}{smondal@memphis.edu}}
\date{}

\begin{document}

\maketitle

\begin{abstract}
  An increasing sequence $(a_n)$ of positive integers which satisfies
  $\frac{a_{n+1}}{a_n}>1+\eta$ for some positive $\eta$ is called a
  lacunary sequence.  It has been known for over twenty years that
  every lacunary sequence is strong sweeping out which means that in
  every aperiodic dynamical system we can find a set $E$ of arbitrary
  small measure so that
  $\limsup_N\frac1N \sum_{n\le N}\setone_E(T^nx)=1$ and
  $\liminf_N\frac1N \sum_{n\le N}\setone_E(T^nx)=0$ almost everywhere.
  In this paper we improve this result by showing that if $(a_n)$
  satisfies only $\frac{a_{n+1}}{a_n}>1+\frac1{(\log\log n)^{1-\eta}}$
  for some positive $\eta$ then it is already strong sweeping out.
\end{abstract}
\tableofcontents
\section{Introduction and main results:}
Throughout this paper, we use the notation
\begin{equation}
  \label{eq:6}
  [N]\cq \cbrace*{1,2,\dots,N}
\end{equation}
Let $T$ be a measure preserving transformation on the probability
space $(X, \Sigma, \mu)$.  After Birkhoff's pointwise ergodic theorem
was proved, naturally the question was raised whether it is possible
to generalize the theorem along any subsequence $(a_n)$ of integers
instead of taking the entire sequence $(n)$.  Krengel\autocite{K} was
the one who first constructed a strictly increasing sequence $(a_n)$
of positive integers so that in every aperiodic system the ergodic
averages $\frac{1}{N}\displaystyle\sum_{n\in[N]}f(T^{a_n}x)$ diverge
almost everywhere.  Soon after, Bellow\autocite{Bellow} showed that if
$(a_n)$ is a \emph{lacunary sequence}, that is, it satisfies
$\frac{a_{n+1}}{a_n}\geq 1+\eta$ for some positive $\eta$, then the
ergodic averages along $(a_n)$ diverge a.e. in every aperiodic system,
though she showed divergence only for $f\in L^p$ for $p<\infty$.  Our
first main result provides a growth condition for a sequence to be
\emph{ pointwise bad } which applies to \emph{non lacunary sequences}
as well, that is sequences for which $\lim_n \frac{a_{n+1}}{a_n}=1$.

\begin{thm}[label={thm:1}]{Strong sweeping out with deterministic
    condition}{}\rr
  Suppose $(a_n)$ is a sequence which satisfies
  \begin{equation}
    \label{eq:2}
    \frac{a_{n+1}}{a_n}\geq 1+ {\frac1{(\log \log n)^{1-\eta}}}
  \end{equation}
  for some $\eta>0$.

  Then in every aperiodic dynamical system $(X,\Sigma, \mu, T)$ and
  every $\epsilon>0$, there exists a set $E\in \Sigma$ with
  $\mu(E)<\epsilon$ such that for almost every $x\in X$, we have
  \begin{align}
    \label{eq:4}
    \limsup_{N\to
    \infty}\frac{1}{N}\sum_{n\in[N]}\setone_E (T^{a_n}x)&= 1\\
    \intertext{ and }
    \liminf_{N\to \infty}\frac{1}{N}\sum_{n\in[N]}\setone_E (T^{a_n}x)&= 0
  \end{align}
\end{thm}

An example of a sequence $(a_n)$ which satisfies the growth condition
in \cref{eq:2} is
$a_n= \intpart{e^{\frac{n}{(\log \log n)^{1-\eta}}}}$, for some
$\eta>0$.

\begin{defn}[label={defn:3}]{Pointwise good and bad sequences}{}\rr
  Let $1\leq p\leq \infty$ and $(a_n)$ be a sequence of positive
  integers.

  We say $(a_n)$ is \emph{pointwise good} for $L^p$ if for every
  measure preserving system $(X,\Sigma, \mu, T)$ and every
  $f\in L^p(X)$,
  $\displaystyle\lim_{N\to \infty}\frac{1}{N}\sum_{n\in
    [N]}f(T^{a_n}x)$ exists a.e.

  We say the sequence $(a_n)$ is \emph{pointwise bad} for $L^p$ if for
  every aperiodic measure preserving system $(X,\Sigma, \mu, T)$ there
  exists a function $f\in L^p(X)$ such that
  $\displaystyle\lim_{N\to \infty}\frac{1}{N}\sum_{n\in
    [N]}f(T^{a_n}x)$ fails to exist a.e.
\end{defn}
Whether a sequence will be pointwise good or bad for $L^p$ depends on
many factors, such as, the speed of the sequence $(a_n)$, the value of
$p$, and sometimes the intrinsic arithmetic properties of the sequence
$(a_n)$. From \Cref{thm:1}, we can see that if a sequence grows very
fast, then it will be pointwise bad even for $L^\infty$. On the other
extreme, if a sequence grows slower than any positive power of $n$,
for example $(a_n)=\big((\log n)^c\big)$ for some $c>0$, then it is
again pointwise bad for $L^\infty$ as shown in Jones and Wierdl
\autocite[Example 2.18]{JW}.  See also a
related result of Loyd\autocite{Loyd}. Bellow\autocite{Bellow2} and
Reinhold-Larsson\autocite{Karin} proved that whether a sequence will
be pointwise good for $L^p$ or not can depend on the value of
$p$. More precisely, they showed that for any given
$1\leq p<q\leq \infty$, there are sequences $(a_n)$ which are
pointwise good for $L^q$ but pointwise bad for $L^p$.
Parrish\autocite{Andrew} gives refinements of these results in terms
of Orlicz spaces.
In general, neither the growth rate of the sequence, nor the value of $p$ alone determines whether the sequence is pointwise good or bad.
In some cases, one has to analyze the intrinsic arithmetic
properties of the sequence $(a_n)$. One such curious example is
$(n^k)_n$.  A celebrated result of Bourgain\autocite[Theorem 2]{BO} says
that the sequence $(n^k)_n$ is pointwise good for $L^2$ when $k$ is a
positive integer. On the other hand, the sequence
$\pa*{\intpart{n^k+\log n}}$ is known to be pointwise bad for $L^2$
when $k$ is a positive integer\autocite[Theorem C]{BKQW}.

After showing that polynomials are pointwise good for $L^2$, Bourgain
showed\autocite{MR1019960} that a polynomial sequence is pointwise
good for $L^p$ for every $p>1$ and Wierdl\autocite{Wierdl} showed the
same for the sequence of primes.  However, the sequence of squares is
pointwise bad for $L^1$, as was shown by
Buczolich-Mauldin\autocite{BuczolichMauldin}, and
LaVictoire\autocite{LaVictoire2} showed the same for the sequence
$(n^k)$ of $k$th powers for a fixed positive integer $k$ and for the
sequence of primes.  It was largely believed that there cannot be any
sequence $(a_n)$ which is pointwise good for $L^1$ and satisfies
$(a_{n+1}-a_n) \to \infty$ as $n \to \infty$, but
Buczolich\autocite{Buczolich} disproved this conjecture.
LaVictoire\autocite{LaVictoire1} showed that a large class or random
sequences also serve as counterexamples.

Urban and Zienkiewicz\autocite{UrbanZien} showed that
$\lfloor n^c\rfloor, c\in(1,1.001)$ is pointwise good for $L^1$. The
current best result is due to Mirek\autocite{Mirek} who showed that
$\lfloor n^c\rfloor, c\in(1,\frac{30}{29})$ is pointwise good for
$L^1$, see also \autocite{Trojan}. It would be interesting to know if
the latter result can be extended to all positive non integer $c$.
The case of $L^2$ is known\autocite{BKQW} as well as $L^p$, $p>1$.

For further exposition in this area, the reader is referred to the
survey article of \autocite{RW}.

The extreme case scenario of non-convergence is the strong sweeping
out property of a sequence.
\begin{defn}[label={defn:1}]{strong sweeping out}{}\rr
  Let $(a_n)$ be a sequence of integers.

  We say $(a_n)$ is \emph{strong sweeping out} if in every aperiodic
  dynamical system $(X,\Sigma, \mu, T)$ and every $\epsilon>0$, there
  exists a set $E\in \Sigma$ with $\mu(E)<\epsilon$ such that for
  almost every $x\in X$ we have
  \begin{align}
    \label{eq:5}
    \limsup_{N\to
    \infty}\frac{1}{N}\sum_{n\in[N]}\setone_E (T^{a_n}x)&= 1\\
    \intertext{ and }
    \liminf_{N\to \infty}\frac{1}{N}\sum_{n\in[N]}\setone_E (T^{a_n}x)&= 0
  \end{align}
\end{defn}
Akcoglu etal\autocite{ABJLRW} showed that every lacunary sequence is
strong sweeping out, and our \Cref{thm:1} extends their result.

\Cref{thm:1} can also be interpreted from a different viewpoint.
Jones-Wierdl\autocite[Corollary 2.14]{JW} proved that for
$1\leq p < \infty$ if a sequence $(a_n)$ satisfies
$\dfrac{a_{n+1}}{a_n}\geq 1+{\frac{1}{(\log n)^{\frac{1}{p}-\eta}}}$
for some $\eta>0$, then $(a_n)$ is pointwise bad, and our \cref{thm:1}
can be viewed as an extension of this result to not only $L^\infty$
but indicators as well.  It is also interesting to compare our result
with the results of Berkes\autocite{Berkes} on lacunary polynomials.

In the next theorem, we will give a probabilistic condition for a
sequence to be strong sweeping out. Before we state the result, let us
explain the notion of a randomly generated sequence. Suppose
$(\sigma_n)$ is a sequence of positive numbers; $\sigma_n$ is the
probability with which $n$ is chosen into the random sequence. More
precisely, let $Y_n$ be a sequence of $0-1$ valued random variables on
the probability space $(\Omega,\beta, P)$ so that $P(Y_n=1)=\sigma_n$
and $P(Y_n=0)=1-\sigma_n.$ For each $\omega\in \Omega$, let $A^\omega$
be the sequence defined by the property that $n\in A^\omega$ if and
only if $Y_n(\omega)=1.$ Our second main result is the following:
\begin{thm}[label={thm:3}]{Strong sweeping out with probabilistic
    condition}{}\rr
  Let $\eta>0$ be arbitrary and
  $\sigma_n=\frac{(\log \log \log n)^{1-\eta}}{n}$.

  Then for a.e. $\omega$ the random sequence $A^\omega= (a_n(\omega))$
  is strong sweeping out.
\end{thm}
This result is an improvement of a result of
Jones-Lacey-Wierdl\autocite[Theorem C]{JLW} where they reach the same
conclusion under the stronger hypothesis $\sigma(n)=\frac{1}{n}$.
\section{Proof of the main results:}
Let $S$ be a finite set and let $f=(f_n)_{n\in S}$ be a sequence of
numbers or functions indexed by $S$.  We denote the arithmetic average
of $(f_n)$ by $\setA_{S}f$,
\begin{equation}
  \label{eq:7}
  \setA_Sf=\setA_{n\in S}f_n\cq \frac{1}{\#S}\sum_{n\in S}f_n
\end{equation}
For a sequence $w=(w(n))_{n\in S}$ of numbers not identically $0$,
which we regard as the sequence of weights, we denote the $w$-weighted
average of $f$ by $\setA^w_{S}f$,
\begin{equation}
  \label{eq:8}
  \setA^w_Sf=\setA^w_{n\in S}f_n\cq \frac{1}{\sum_{n\in S}w(n)}\sum_{n\in S}w(n)f_n
\end{equation}

We will consider higher dimensional torus for proving our
result. Originally, such argument was used by Jones\autocite{JO} to
prove that the finite union of lacunary sequences is strong sweeping
out. Later an extension of this method was used by
Mondal\autocite{M}. We call this technique as the \emph{grid method}.

To prove \cref{thm:1}, we need the following two lemmas:
\begin{lem}[label={lem:1}]{}{}\rr
  Let $\tilde{A}=(a_n)_{n\in[N]}$ be a finite sequence of integers
  which satisfies $\frac{a_{n+1}}{a_n}>2Q$ for some $Q\leq N$. Suppose
  that $\tilde{A}=\cdot\hspace{-13pt}\bigcup\limits_{q\in [Q]}A_{q}$
  be partition of $\tilde{A}$.

  Then there exists an irrational number $r\in[0,1]$ such that for all
  $q\in [Q]$ we have $ra_n\in I_{Q-q}$ (mod $1$) whenever
  $a_n\in A_q$, where $I_q=\big(\frac{q-1}{Q}, \frac{q}{Q}\big).$
\end{lem}
\begin{proof}
  This lemma in a bit different form appeares elsewhere\autocite[Lemma
  2.13]{JW}, hence its proof is skipped here.
\end{proof}
\begin{lem}[label={lem:conze}]{}{}\rr
  Let $A=(a_n)$ be a sequence of integers which satisfies the
  following property:
 
  for every $C>0$, $\epsilon>0$ and $N_1\in \setN$, there exists a
  dynamical system $(X,\Sigma, \mu,T)$, a set $E\in \Sigma$ with
  $\mu(E)<\epsilon$, and an integer $N_2>N_1$ such that
  \begin{equation}
    \label{eq:denialinr}
    \mu\Big\{x\in X: \max_{N_2\leq N \leq N_2}\setA_{n\in [N]}\setone_E (T^{a_n}x)>1-\epsilon\Big\}\geq C\mu(E).
  \end{equation}
  Then the sequence $A=(a_n)$ is strong sweeping out.
\end{lem}
\begin{proof}
  A version of this lemma appears elsewhere\autocite[Theorem 3.1]{M}
  already with detailed proof, so we will just outline the proof here.

  First, by using Calderon's transference
  principle\autocite{Calderon}, one can prove that if we have a
  maximal inequality on $\setR$ (with respect to the transformation
  $\tau(n)=n+1$), then the maximal inequality transfers to any
  dynamical system with the same constant. Hence the hypothesis of
  \Cref{lem:conze} implies that there is a denial of a maximal
  inequality on $\setR$. Now, we invoke Akcoglu etal's
  result\autocite[Theorem 2.3]{ABJLRW} to finish the proof of this
  theorem.
\end{proof}
\begin{proof}[Proof of \Cref{thm:1}]

  Let $C>0$, $N_1\in \setN$, and $\epsilon>0$. By \cref{lem:conze}, it
  will be sufficient to find a set $E$ in $\setT^K$ with
  $\lambda^{(K)}(E)<\epsilon$, and an integer $N_2>N_1$ which satisfy
  \cref{eq:denialinr}. Here $\setT^K$ denotes the $K$-dimensional
  torus and $\lambda^{(K)}$ means the Haar-Lebesgue measure on
  $\setT^K$.
  
  Let $N> N_1$ be a very large positive integer. Since,
  \begin{equation*}
    a_n \neq a_m \text { if } n\neq m
  \end{equation*}
  so $(a_n)$ can also be considered as a set.  Since, for
  $x\in (0,1)$, we have $e^x= 1+x+O(x^2)$, we can rewrite the given
  condition as
  \begin{equation}
    \frac{a_{n+1}}{a_n}> e^\frac{1}{(\log\log N)^{1-\eta}} \text{ for }n\in [N].
  \end{equation}
  Let us choose a natural number $Q=Q(N)$ which just needs to go to
  $\infty$ as $N\to \infty.$ Choose another integer $K=K(N)$ large
  enough so that
  \begin{equation}
    \label{eq:223}
    e^{\frac{K}{(\log\log N)^{1-\eta}}}>2Q.
  \end{equation}
  This implies
  \begin{equation}
    \label{eq:24}
    \frac{a_{n+K}}{a_n}>2Q \text{ for } n\in[N].
  \end{equation}
  For any $k\in [K]$, define
  \begin{equation}
    \label{eq:25}
    A_k:=\big\{a_n:n\equiv k\ (\text{mod }K) \text{ and } n\in[N] \big\}.
  \end{equation}
  Observe that for each $k\in [K]$, $A_k$ satisfies the hypothesis of
  \cref{lem:1}.  That means each $A_k$ has the property that if
  it is partitioned into $Q$ sets, so
  $A_k =\displaystyle\cdot\hspace{-13pt}\bigcup_{q \in [Q]}A_{k,q}$
  then there is an irrational number $r_k$ so that
  \begin{equation}
    \label{eq:14}
    r_k a_n \in I_{Q-q} \ \ (\text{mod } 1) \text{ for all } a_n \in A_{k,q},\  q\in [Q],
  \end{equation}
  where
  \begin{equation}
    I_q:=\big(\frac{q-1}{Q}, \frac{q}{Q}\big).
  \end{equation}

  The above partition $A_1, A_2, . . . , A_K$ of $(a_n)_{n\in[N]}$
  naturally induces a partition of the index set $[N]$ into $K$ index
  sets $\mathcal N_k, k\in [K]$.  We then have

  \begin{equation}
    \label{eq:13}
    \setA_{n\in J} f(T^{a_n}x)=\frac{1}{\#J}\sum_{k\in [K]}\sum_{n \in J\cap \mathcal {N}_k} f(T^{a_n}x).
  \end{equation}

  The space of action is the $K$ dimensional torus ${\mathbb T}^K$,
  subdivided into little $K$ dimensional cubes $C$ of the form
  \begin{equation}
    \label{eq:15}
    C=I_{q(1)} \times I_{q(2)} \times......\times I_{q(k)} \text{ for some } {q(k)}\leq Q \text{ for } q\leq Q.
  \end{equation}

  At this point it is useful to introduce the following vectorial
  notation to describe these cubes $C$. For a vector
  $\mathbf {q}=\big(q(1), q(2), . . . , q(K)\big)$ with
  $q(k) \in [Q]$, define
  \begin{equation}
    \label{eq:16}
    I_{\mathbf {q}} := I_{q(1)} \times I_{q(2)} \times \dots \times  I_{q(K)}.
  \end{equation}

  Since each component $q(k)$ can take up the values $1,2, \dots ,Q$,
  we divided $\setT^Q$ into $Q^k$ cubes.  We also consider the ``bad''
  set $E$ defined by

  \begin{equation}
    \label{eq:17}
    E:= \cup_{k\leq K} (0,1)\times(0,1)\times\dots \times \underbrace{\big(I_1\cup I_2\big)}_\text{k-th coordinate}\times\dots \times (0,1).
  \end{equation}

  Defining the set $E_k$ by
  \begin{equation}
    \label{eq:18}
    E_k:= (0,1)\times (0,1) \times \dots \times \underbrace{\big(I_1\cup I_2\big)}_\text{k-th coordinate}\times \dots \times (0,1)
  \end{equation}
  we have
  \begin{equation}
    \label{eq:119}
    E=\cup_{k\leq K}E_k  \text{ and } \lambda^{(K)}(E_k)\leq \frac{2}{Q} \text{ for every } k\leq K 
  \end{equation}

  By \cref{eq:119}, we have
  \begin{equation}
    \lambda^{(K)}(E)\leq \frac{2K}{Q}.
  \end{equation}

  Since we want the measure of the ``bad set'' to be smaller and
  smaller, we must assume
  \begin{equation}
    \label{eq:121}
    K<< Q.
  \end{equation}

  Now the idea is to have averages that move each of the cubes into
  the support of the set $E$.  The $2$-dimensional version of the
  process is illustrated in \Cref{fig:grid}.
  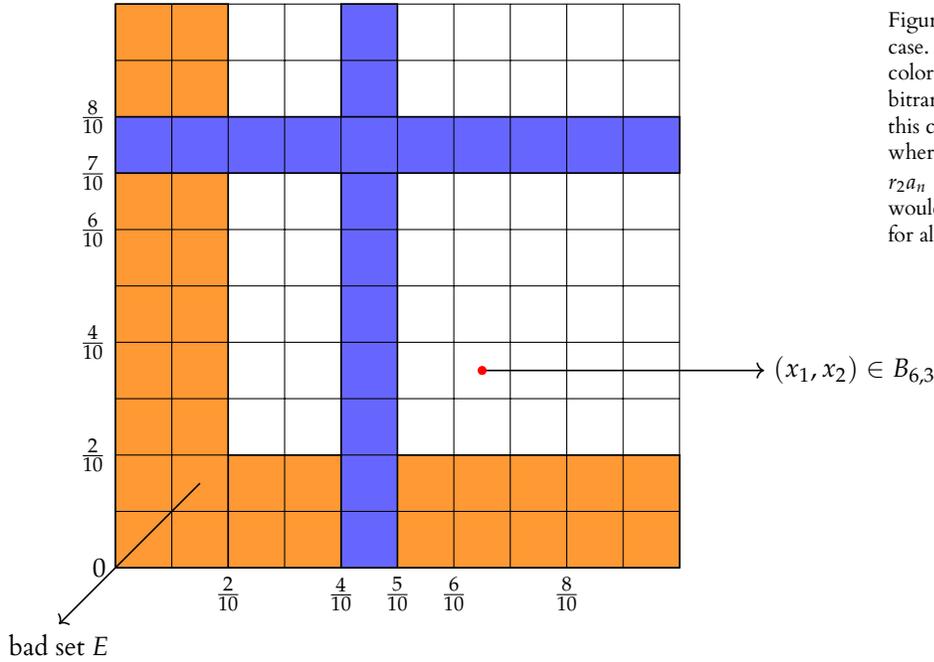
\begin{figure}[h!]
    
    \begin{tikzpicture}[scale=.75]
		
  \tkzDefPoints{0/0/O, 10/10/M}
  
  \tkzDefPoints{0/2/f2, 0/4/f4, 0/6/f6, 0/7/f7, 0/8/f8, 0/10/f10,
    2/0/h2, 4/0/h4, 5/0/h5, 6/0/h6, 8/0/h8, 10/0/h10}

  \tkzDefPoints{4/10/p4, 5/10/p5, 10/7/p7, 10/8/p8}
  \tkzDefPoints{10/2/g102, 2/10/g210}
  \tkzDefPoints{1.5/1.5/es, 6.5/3.5/bs}
  \tkzDefPoints{-1/-1/ee, 11.5/3.5/be}

  \tkzLabelPoint[left](O){$0$}
  \tkzLabelPoint[left](f2){$\frac2{10}$}
  \tkzLabelPoint[left](f4){$\frac4{10}$}
  \tkzLabelPoint[left](f6){$\frac6{10}$}
  \tkzLabelPoint[left](f7){$\frac7{10}$}
  \tkzLabelPoint[left](f8){$\frac8{10}$}

  \tkzLabelPoint[below](h2){$\frac2{10}$}
  \tkzLabelPoint[below](h4){$\frac4{10}$}
  \tkzLabelPoint[below](h5){$\frac5{10}$}
  \tkzLabelPoint[below](h6){$\frac6{10}$}
  \tkzLabelPoint[below](h8){$\frac8{10}$}

  \tkzDrawPolygon[fill=orange!80](O,h10,g102,f2)
  \tkzDrawPolygon[fill=orange!80](O,f10,g210,h2)
  \tkzDrawPolygon[fill=blue!60](h4,h5,p5,p4)
  \tkzDrawPolygon[fill=blue!60](f7,f8,p8,p7)
  \draw (O) grid (M);

  \tkzDrawSegments[->](es,ee bs,be)
  \tkzLabelPoint[below](ee){bad set $E$}
  \tkzLabelPoint[right](be){$(x_1,x_2)\in B_{6,3}$}
  \tkzDrawPoint[red](bs)
\end{tikzpicture}
    \caption[-5cm]{Illustration of the 2-dimensional case.  Here the
      ``bad set'' E is the orange colored region. Let $(x_1,x_2)$ be an
      arbitrary point (which belongs to $B_{6,3}$ in this case). We
      need to look at an average where
      $r_1a_n\in (\frac{4}{10},\frac{5}{10})$ for all $n\in {A_1}$ and
      $r_2a_n\in (\frac{7}{10},\frac{8}{10})$ for all $n\in {A_2}.$
      Then it would give us $(x_1,x_2)+(r_1a_n,r_2a_n)\in E$ for all
      $n\in {A_1}\cup {A_2}.$ }
    \label{fig:grid}
  \end{figure}

  Since we have $Q^K$ cubes, we need to have $Q^K$ averages
  $\setA_{J_i}$. This means we need to have $Q^K$ disjoint intervals
  $J_i$ of indices. The length of these intervals $J_i$ needs to be
  ``significant'', in comparison with $J_{i-1}$. For our purpose,
  $J_i = (2^{N_1+i}, 2^{N_1+i+1})$ will be suitable. This means we
  need to have $Q^K$ exponents to be available, which implies that
  \begin{equation}
    \label{eq:23}
    N\geq 2 ^{Q^K}.
  \end{equation}
  For simplicity, we assume that
  \begin{equation}
    \label{eq:124}
    N=2^{Q^K}
  \end{equation}
  To make our plan work, first let us check that we can really choose
  such $K(N)$ which satisfies the condition \cref{eq:121},
  \cref{eq:124} and \cref{eq:223}.  Let us write \cref{eq:223} as
  \begin{equation}
    \label{eq:126}
    \frac{K}{(\log\log N)^{1-\eta}}> \log 2Q.
  \end{equation}
  Using the assumption $N=2^{Q^K}$, the condition in \cref{eq:126}
  becomes
  \begin{equation}
    \label{eq:27}
    \frac{K}{(K\log Q)^{1-\eta}}> \log 2Q.
  \end{equation}
  After rearranging and ignoring the difference between $\log Q$ and
  $\log 2Q$, we get
  \begin{equation}
    \label{eq:28}
    K^\eta>\big(\log Q\big)^{2-\eta},
  \end{equation}
  which we can simplify a bit generously to
  \begin{equation}
    \label{eq:29}
    K>(\log Q)^\frac{2}{\eta}.
  \end{equation}
  We certainly can choose $K$ to satisfy eq. \cref{eq:29} and also
  make sure that $\frac{K(N)}{Q(N)}$ goes to $0$ and $N\to \infty.$
  Choose $N_2$ large enough so that $K(N_2)$ and $Q(N_2)$ satisfy
  \cref{eq:29} and the following:
  \begin{equation}
    \frac{2K(N_2)}{Q(N_2)}<\min\{\epsilon,\frac1{C}\}.
  \end{equation}
  So we have $Q^K$ cubes $C_i, i \in [Q^K]$ and $Q^K$ intervals
  $J_i,i \in [Q^K]$. We match $C_i$ with $J_i$. We know that $C_i$ is
  of the form
  \begin{equation}
    \label{eq:125}
    C_i= I_{\mathsf{q}_i},
  \end{equation} 
  for some $K$ dimensional vector
  $\mathsf{q}_i= \big(q_i(1),q_i(2),\dots q_i(K)\big) $ with
  $q_i(k)\in [Q]$ for every $k\in [K].$ The interval $J_i$ is
  partitioned as
  \begin{equation}
    \label{eq:26}
    J_i=\displaystyle \bigcup_{k\in[K]}(J_i\cap \mathcal{N}_k).
  \end{equation}
  For a given $k\in [K]$, let us define the set of indices
  $\mathcal{N}_{k,q}$ for $q\leq Q$, by
  \begin{equation}
    \label{eq:127}
    \mathcal{N}_{k,q}:= \displaystyle\bigcup_{i\leq Q^K, q_i(k)=q} (J_i\cap \mathcal{N}_k)
  \end{equation}
  Since the sets $A_{k,q}:= \big\{a_n: n\in \mathcal{N}_{k,q}\big\}$
  form a partition of $A_k$, by the argument above \cref{eq:14}, there
  is a real number $r_k$ so that
  \begin{equation}
    \label{eq:228}
    r_k a_n\in I_{Q-q}\ \ (\text{ mod }1) \text{ for } n\in \mathcal{N}_{k,q} \text{ and } q\leq Q.
  \end{equation}

  Define the transformation $T$ on the $K$ dimensional torus $\setT^K$
  by
  \begin{equation}
    \label{eq:116}
    T(x_1,x_2,\dots, x_K):= (x_1+r_1t,x_2+r_2t,\dots , x_K+r_Kt).
  \end{equation}
  We claim that
  \begin{equation}
    \label{eq:117}
    \Big\{x|\max_{i\in [Q^K]} \setA_{n \in J_i}\setone_E(T^{a_n}x)=1\Big\}= \setT ^K 
  \end{equation}

  Indeed, let $x\in C_i$ and consider the averages $\setA_{J_i}$. Let
  us write
  \begin{equation}
    \label{eq:118}
    \setA_{n \in J_i}\setone_E(T^{a_n}x)=\frac{1}{\# J_i}\sum_{k\leq {K}}\sum_{n\in J_i\cap \mathcal{N}_{k}}\setone_E(T^{a_n} {x}) 
  \end{equation}
  \begin{equation}
    \label{eq:19}
    =\frac{1}{\# J_i}\sum_{k\leq {K}}\sum_{n\in J_i\cap \mathcal{N}_{k}}\setone_E(x_1+r_1a_n,x_2+r_2a_n,\dots ,x_K+r_Ka_n).
  \end{equation}
  We claim that for each $k\in [K]$
  \begin{equation}
    \label{eq:20}
    (x_1+r_1a_n,x_2+r_2a_n,\dots x_k+r_ka_n,\dots, x_K+r_Ka_n)\in  E_k \text{ if }n\in J_i\cap \mathcal{N}_k
  \end{equation}

  Since, $E_k \subset E$, we would have
  \begin{equation}
    \setone_E(x_1+r_1a_n,x_2+r_2a_n,\dots x_k+r_ka_n,\dots, x_K+r_Ka_n)=1
    \text { if } n\in J_i\cap \mathcal{N}_k \end{equation}
  which would mean
  \begin{equation}
    \setone _E(T^{a_n}x)=1 \text{ for all } n\in J_i.
  \end{equation}

  So let us prove \cref{eq:20}.  Since $x\in C_i= I_{\mathsf{q}}$ we
  have $x_k\in I_{q_i(k)}$ for every $k$. By the definition of $r_k$
  in eq. (\cref{eq:228}) we have $r_ka_n\in I_{Q-q_i(k)}$ if
  $n\in J_i\cap \mathcal{N}_k$.  It follows that
  \begin{equation}
    x_k+r_ka_n\in I_{q_{i(k)}}+I_{Q-q_i(k)} \text{ if } n\in J_i\cap\mathcal{N}_k.
  \end{equation}
  Since $I_{q_{i(k)}}+I_{Q-q_{i(k)}}\subset I_1\cup I_2$, we get
  \begin{equation*}
    x_k+r_ka_n\in I_1\cup I_2 \text{ if } n\in J_i\cap \mathcal{N}_k.
  \end{equation*}
  By the definition of $E_k$ in \cref{eq:18}, this implies that

\begin{equation*}
  (x_1+r_1a_n,x_2+r_2a_n,\dots x_k+r_ka_n,\dots, x_K+r_Ka_n)\in  E_k 
\end{equation*}
as claimed.
\end{proof}

\begin{remarks}
  \begin{enumerate}
  \item Sharpness: We can prove \Cref{thm:1} by replacing the
    assumption \cref{eq:2} with
    $\frac{a_{n+1}}{a_n}>1+\frac{(\log\log\log n)^{2+\eta}}{\log\log
      n}$, $\eta>0.$
  \item \Cref{thm:1} can also be proved by using
    \autocite[Theorem 3.1]{PS}. By applying this result, one can
    slightly weaken the hypothesis. More precisely, we can prove that
    any sequence $(a_n)$ satisfying
    $\frac{a_{n+1}}{a_n}>1+\frac{(\log\log\log n)^{1+\eta}}{\log\log
      n}$, $\eta>0$ is strong sweeping out.
  \end{enumerate}
\end{remarks}

We can generalize \Cref{thm:1} for weighted ergodic averages
in the following way:

 \begin{thm}[label={thm:4}]{}{}\rr
   Let $(w(n))$ be a sequence of real numbers from the interval
   $(0,1]$ and denote $G(n):= \sum_{i\in [n]} w(i)$.  
   Suppose $(a_n)$ is a sequence of integer which satisfies
   \begin{equation}
     \label{eq:42}
     \frac{a_{n+1}}{a_n}\geq 1+ {\frac1{(\log \log G(n))^{1-\eta}}}
   \end{equation}
   for some $\eta>0$. Then $(a_n)$ satisfies the strong sweeping out
   property for the $w$-weighted averages $\setA^w_{n\in [N]}f(T^{a_n}x)$.
	
 \end{thm}
 \begin{proof}
   The proof of \Cref{thm:1} will work in this case after the
   following obvious modifications.  Because of changing the summation
   method, instead of $\setA_J$, we have to work with the averages
   $\setA^w_J$.  Accordingly, the length of $J_i$ also has to be
   changed. A suitable choice of $J_i$ in this case would be
   following:
   \begin{equation}
     J_i=\Big(G^{-1}(2^{N_1+i}),G^{-1}(2^{N_1+i+1})\Big)
   \end{equation}
   Now, one can reiterate the argument given in \Cref{thm:1} to get
   the desired conclusion.
 \end{proof}
 
 \begin{cor}[label={cor:1}]{}{}\rr
   If $(a_n) $ is a sequence of integers satisfying
   $\frac{a_{n+1}}{a_n}\geq 1+{\frac1{(\log \log \log n)^{1-\eta}}}$
   for some $\eta>0$, then $({a_n})$ is strong sweeping out for the
   logarithmic averages
   $\setA^{1/n}_{n\in [N]}f(T^{a_n}x)= \dfrac{1}{\log N} \sum_{n\in
     [N]}\frac{1}{n}f(T^{a_n}x)$.
 \end{cor}
 Now, we will prove \cref{thm:3}.  For any sequence $A$ of integers,
 we define $A(t):=\{n\in A: n\leq t\}.$ Observe that
 \Cref{thm:3} will follow from \Cref{thm:1} and the following
 lemma.

\begin{lem}[label={lem:5}]{}{}\rr
  Under the hypothesis of \Cref{thm:3}, for a.e. $\omega$,
  there exists a subsequence $B^\omega=(b_n(\omega))$ of
  $A^\omega=(a_n(\omega))$ such that the following holds:
  \begin{align}
    \label{eq:61}
    \lim_{t\to \infty} \frac{B^\omega(t)}{A^\omega(t)}&=1 \\
    \label{eq:62}
    \frac{b_{n+1}(\omega)}{b_n(\omega)}&> e^{\frac{1}{(\log \log n)^{1-{\eta/2}}}}
  \end{align}
\end{lem}
\begin{proof}
  Let $u_n= \min \{t| \displaystyle\sum_{k\leq t}\sigma_k \geq
  n\}$. First observe that $u_n\sim e^{n(\log \log n)^{-1+\eta}}. $

  By the strong law of large numbers, we have for a.e. $\omega$,

  $\displaystyle\lim_{n\to \infty}
  \frac{A^\omega(u_n)}{\displaystyle\sum_{u\leq u_n} \sigma_u}=1
  \text{ which implies that\ \ } \lim_{n\to \infty}
  \frac{A^\omega(u_n)}{n}=1.$

  Clearly,
  $\frac{u_{n+1}}{u_n}\sim e^\frac{1}{(\log \log n)^{1-\eta}}.$
  However, this does not imply that
  $\frac{a_{n+1}(\omega)}{a_n(\omega)}> e^{\frac{1}{(\log \log
      n)^{1-\eta}}}$.

  \vspace{.5cm}
  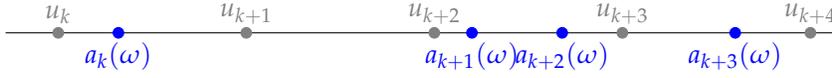
\begin{figure}

    \caption{An example where $A^\omega$ may not satisfy the condition
      \cref{eq:62}}
    \label{fig:1}
    \begin{tikzpicture}
	
      \draw (-6,0) -- (5,0); \filldraw [gray] (-5.3,0) circle (2pt)
      node[anchor=south] {$u_k$}; \filldraw [blue] (-4.5,0) circle
      (2pt) node[anchor=north] {$a_k(\omega)$}; \filldraw [gray]
      (-2.8,0) circle (2pt) node[anchor=south] {$u_{k+1}$}; \filldraw
      [blue] (.2,0) circle (2pt) node[anchor=north]
      {$a_{k+1}(\omega)$}; \filldraw [blue] (1.4,0) circle (2pt)
      node[anchor=north] {$a_{k+2}(\omega)$}; \filldraw [gray]
      (-0.3,0) circle (2pt) node[anchor=south] {$u_{k+2}$}; \filldraw
      [gray] (2.2,0) circle (2pt) node[anchor=south] {$u_{k+3}$};
      \filldraw [blue] (3.7,0) circle (2pt) node[anchor=north]
      {$a_{k+3}(\omega)$}; \filldraw [gray] (4.7,0) circle (2pt)
      node[anchor=south] {$u_{k+4}$};
 \end{tikzpicture}
  \end{figure}
  \vspace{1cm}
 
  So, we need to modify our sequence.

  Let $v_n=e^{n(\log \log n)^{-1+\frac{\eta}{2}}}$ and
  $I_n=[v_n,v_{n+1})\cap \mathbb{N}$.  The properties of $(v_n)$, that
  we shall use here are the following:
 \begin{equation}
    \label{eq:64}
\frac{v_{n+1}}{v_n}\geq e^{\frac{1}{(\log \log n)^{1-\frac{\eta}{2}}}} \text{ and } \lim_{n\to \infty} \sum_{u\in I_n}\sigma_u=0.
  \end{equation}

  Let
  $D^\omega= (d_n(\omega)):= \{d: d\in I_n\cap A^\omega \text{ for
    some } n \in \mathbb {N} \text{ satisfying } I_{n+1}\cap A^\omega
  \neq \emptyset\}$.

  And
  $E^\omega= (e_n(\omega)):= \displaystyle\bigcup_{n\in
    \setN}\{I_n\cap A^\omega: |I_n\cap A^\omega|>1\}$.

  In \Cref{fig:2}, $a_{k+1}(\omega)\in D^{\omega}$ because
  $I_{k+1}\cap A^{\omega} \neq \emptyset$.  And
  $a_{k+3}(\omega), a_{k+4}(\omega) \in E^\omega.$

 \vspace{1cm}
  \begin{figure}

    \begin{tikzpicture}
      \caption{Construction of $D^\omega$ and $E^{\omega}$}
      \label{fig:2}
      \draw (-6,0) -- (5.2,0); \filldraw [gray] (-5.5,0) circle (2pt)
      node[anchor=south] {$v_k$}; \filldraw [teal] (-4.5,0) circle
      (2pt) node[anchor=north] {$a_k(\omega)$}; \filldraw [gray]
      (-4,0) circle (2pt) node[anchor=south] {$v_{k+1}$}; \filldraw
      [red] (-3,0) circle (2pt) node[anchor=north]
      {$a_{k+1}(\omega)$}; \filldraw [teal] (-1.5,0) circle (2pt)
      node[anchor=north] {$a_{k+2}(\omega)$}; \filldraw [gray]
      (-2.5,0) circle (2pt) node[anchor=south] {$v_{k+2}$}; \filldraw
      [gray] (-1,0) circle (2pt) node[anchor=south] {$v_{k+3}$};
      \filldraw [red] (.65,0) circle (2pt) node[anchor=north]
      {$a_{k+3}(\omega)$}; \filldraw [gray] (.5,0) circle (2pt)
      node[anchor=south] {$v_{k+4}$}; \filldraw [gray] (2,0) circle
      (2pt) node[anchor=south] {$v_{k+5}$}; \filldraw [teal] (4.5,0)
      circle (2pt) node[anchor=north] {$a_{k+5}(\omega)$}; \filldraw
      [gray] (3.5,0) circle (2pt) node[anchor=south] {$v_{k+6}$};
      \filldraw [red] (1.9,0) circle (2pt) node[anchor=north]
      {$a_{k+4}(\omega)$}; \filldraw [gray] (5,0) circle (2pt)
      node[anchor=south] {$v_{k+7}$};
	
    \end{tikzpicture}
  \end{figure}
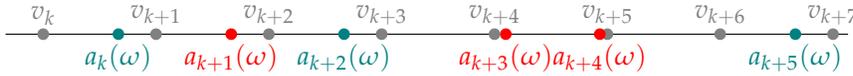
  \vspace{1cm}

  Define $B^\omega:=A^\omega \setminus \ (D^\omega\cup E^\omega)
  $. Note that $B^\omega$ satisfies the following properties
  \begin{enumerate}
  \item If $|B^\omega\cap I_n|= 1$ then $|B^\omega\cap I_{n+1}|=0.$
  \item $|B^\omega\cap I_n|\leq 1$ for all $n$.
  \end{enumerate}
  from which it follows that $(b_n(\omega))$ satisfies \cref{eq:62}.
 
  \vspace{1cm}
  \begin{figure}
 
    \begin{tikzpicture}
      \caption{Construction of $B^\omega$}
      \label{fig:3}
 	
      \draw (-6,0) -- (5.2,0); \filldraw [gray] (-5.5,0) circle (2pt)
      node[anchor=south] {$v_k$}; \filldraw [teal] (-4.5,0) circle
      (2pt) node[anchor=north] {$a_k(\omega)$}; \filldraw [gray]
      (-4,0) circle (2pt) node[anchor=south] {$v_{k+1}$};

      \filldraw [teal] (-1.5,0) circle (2pt) node[anchor=north]
      {$a_{k+2}(\omega)$}; \filldraw [gray] (-2.5,0) circle (2pt)
      node[anchor=south] {$v_{k+2}$}; \filldraw [gray] (-1,0) circle
      (2pt) node[anchor=south] {$v_{k+3}$};
 
      \filldraw [gray] (.5,0) circle (2pt) node[anchor=south]
      {$v_{k+4}$}; \filldraw [gray] (2,0) circle (2pt)
      node[anchor=south] {$v_{k+5}$}; \filldraw [teal] (4.5,0) circle
      (2pt) node[anchor=north] {$a_{k+3}(\omega)$}; \filldraw [gray]
      (3.5,0) circle (2pt) node[anchor=south] {$v_{k+6}$};
 
      \filldraw [gray] (5,0) circle (2pt) node[anchor=south]
      {$v_{k+7}$};
 	
    \end{tikzpicture}
 \end{figure}
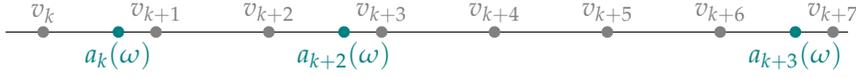
 
It remains to verify \Cref{eq:61}.

  It will be sufficient to show that for a.e. $\omega$,

 \begin{align*}
   \lim _{t\to \infty} \frac{D^\omega(t)}{A^\omega(t)}=0 \text{   and   } \lim _{t\to \infty} \frac{E^\omega(t)}{A^\omega(t)}=0.
 \end{align*}

 By the strong law of large numbers, we need to show that

\begin{align*}
  \lim _{t\to \infty} \frac{\mathbb{E}D^\omega(t)}{\sum_{n\leq t}\sigma_n}=0 \text{   and   } \lim _{t\to \infty} \frac{\mathbb{E}E^\omega(t)}{\sum_{n\leq t}\sigma_n}=0.
\end{align*}

Note that
$D^\omega(t)\leq \displaystyle\sum_{v_k\leq t}\big(\sum_{n\in I_k}
Y_n(\omega)\big).\sup_{n\in I_{k+1}}Y_n$.  Since, the random variables
$X_k=\sum_{n\in I_k}Y_n,k=1,2,3,....,$ are independent, it follows
that
\begin{align*}
  \mathbb{E}D^\omega(t)&\leq \mathbb{E}\sum_{v_k\leq t}\Big(\sum_{n\in I_k}Y_n\Big).\sup_{n\in I_{k+1}}Y_n\\
                       &\leq \mathbb{E}\sum_{v_k\leq t}\Big(\sum_{n\in I_k}Y_n\Big).\Big(\sum_{n\in I_{k+1}}Y_n\Big)\\
                       &\leq \sum_{v_k\leq t}\big(\sum_{n\in I_k}\sigma_n\big).\big(\sum_{n\in I_{k+1}}\sigma_n\big)\\
\end{align*}
Hence,
\begin{align*}
  \limsup_{t\to \infty}\frac{\mathbb{E}D^\omega(t)}{\sum_{n\leq t}\sigma_n}
  &\leq
    \limsup_{t\to \infty} \frac{\sum_{v_k\leq t}\big(\sum_{n\in I_k}\sigma_n\big).\big(\sum_{n\in I_{k+1}}\sigma_n\big)}{\sum_{n\leq t}\sigma_n}\\
  &\leq 	\limsup_{k\to \infty} \Big(\sum_{n\in I_{k+1}}\sigma_n\Big)\frac{\sum_{v_k\leq t}\big(\sum_{n\in I_k}\sigma_n\big)}{\sum_{n\leq t}\sigma_n}\\
  &=0  \text{                   \big( By \cref{eq:64}\big)}
\end{align*} 

Similarly,
\begin{align*}
  &\limsup _{t\to \infty} \frac{\mathbb{E}E^\omega(t)}{\sum_{n\leq t}\sigma_n}\\&=
  \limsup _{t\to \infty} \frac{\sum_{v_k\leq t}\big(2\cdot\sum_{m\neq n\in I_k}\sigma_m\cdot\sigma_n+3\cdot\sum_{m\neq n\neq p\in I_{k}}\sigma_m\cdot\sigma_n\cdot\sigma_p+\dots +l\cdot\sum_{m_1\neq m_2\neq \dots \neq m_l\in I_{k}}\sigma_{m_1}\cdot\sigma_{m_2}\dots \sigma_{m_l} \big)}{\sum_{n\leq t}\sigma_n}\ \ (\text{here} |I_k|=l )\\&
  \leq \limsup _{t\to \infty} \frac{\sum_{v_k\leq t}\Big(2\cdot\big(\sum_{m\in I_k}\sigma_m\big)^2+3\cdot\big(\sum_{m\in I_k}\sigma_m\big)^3+\dots +l\cdot\big(\sum_{m\in I_k}\sigma_m\big)^l \Big)}{\sum_{n\leq t}\sigma_n}\\&
  \leq \limsup _{t\to \infty} \frac{\sum_{v_k\leq t}\big(\sum_{m\in I_k}\sigma_m\big)\cdot\Big(2\cdot\big(\sum_{m\in I_k}\sigma_m\big)+3\cdot\big(\sum_{m\in I_k}\sigma_m\big)^2+\dots +l\cdot\big(\sum_{m\in I_k}\sigma_m\big)^{(l-1)}\Big)}{\sum_{n\leq t}\sigma_n}\\&
  \leq \limsup _{k\to \infty}\Big(2\cdot\big(\sum_{m\in I_k}\sigma_m\big)+3\cdot\big(\sum_{m\in I_k}\sigma_m\big)^2+\dots l\cdot\big(\sum_{m\in I_k}\sigma_m\big)^{(l-1)} \Big)\\&
  \leq \limsup _{k\to \infty} x\cdot\frac{(2-x)}{(1-x)^2} \text{     \Big( where $x=\sum_{m\in I_k}\sigma_m$ \Big)}\\&
  =0 \text{                   \big( By \cref{eq:64}\big)}
\end{align*}

This completes the proof.
\end{proof}
\section{Open Problems}
The first problem asks if our result in \Cref{thm:1} is sharp.
\begin{problem}{}{}\rr
  Suppose the sequence $(a_n)$ of positive integers satisfies
  \begin{equation}
    \label{eq:1}
    \frac{a_{n+1}}{a_{n}}>1+\frac1{\log\log n}
  \end{equation}
  Is $(a_n)$ strong sweeping out?
\end{problem}
It is known\autocite{JLW} that there is a pointwise good sequence
$(a_n)$ for $L^2$ satisfying
$\dfrac{a_{n+1}}{a_n}\geq 1+{\frac{1}{(\log n)^{1+\eta}}}$ for every
$\eta>0$ and large enough $n$. We already mentioned\autocite[Corollary
2.14]{JW} that if $(a_n)$ satisfies
$\dfrac{a_{n+1}}{a_n}\geq 1+{\frac{1}{(\log n)^{1/2-\eta}}}$ for some
positive $\eta$ then $(a_n)$ is pointwise bad for $L^2$.

\begin{problem}{}{}\rr
  Suppose the sequence $(a_n)$ of positive integers satisfies
  \begin{equation}
    \label{eq:3}
    \dfrac{a_{n+1}}{a_n}\geq 1+{\frac{1}{\log n}}
  \end{equation}
  for every large enough $n$.

  Is $(a_n)$ then pointwise bad for $L^2$?
\end{problem}
\printbibliography
\end{document}